\documentclass{birkjour}
 \newtheorem{thm}{Theorem}[section]
 
 \newtheorem{lem}[thm]{Lemma}
 
 \theoremstyle{definition}
 \newtheorem{defn}[thm]{Definition}
 \theoremstyle{remark}
 \newtheorem{rem}[thm]{Remark}
 
 \numberwithin{equation}{section}

\usepackage{amssymb}
\usepackage{graphicx}
\usepackage{amsmath}
\usepackage{latexsym}
\usepackage{pstricks}
\usepackage{graphics,color}
\usepackage{color}
\usepackage{mathrsfs}

\newcommand{\om}{\Omega}
\newcommand{\pom}{\partial\Omega}

\newcommand{\bfK}{\mbox{\boldmath{$K$}}}

\newcommand{\bfx}{\mbox{\boldmath{$x$}}}
\newcommand{\bfn}{\mbox{\boldmath{$n$}}}

\newcommand{\bfF}{\mbox{\boldmath{$F$}}}
\newcommand{\bfg}{\mbox{\boldmath{$g$}}}

\newcommand{\bfz}{\mbox{\boldmath{$z$}}}
\newcommand{\bfh}{\mbox{\boldmath{$h$}}}
\newcommand{\bfu}{\mbox{\boldmath{$u$}}}
\newcommand{\bfw}{\mbox{\boldmath{$w$}}}
\newcommand{\bfv}{\mbox{\boldmath{$v$}}}

\newcommand{\bfB}{\mbox{\boldmath{$B$}}}
\newcommand{\bfR}{\mbox{\boldmath{$R$}}}

\newcommand{\bfe}{\mbox{\boldmath{$e$}}}

\newcommand{\bfchi}{\mbox{\boldmath{$\chi$}}}

\newcommand{\bfphi}{\mbox{\boldmath{$\phi$}}}

\newcommand{\bfpsi}{\mbox{\boldmath{$\psi$}}}
\newcommand{\bfxi}{\mbox{\boldmath{$\xi$}}}

\newcommand{\bfvarphi}{\mbox{\boldmath{$\varphi$}}}

\begin{document}

%
%
%
%
%
%
%
%
%

\title[Doubly Nonlinear Parabolic Systems]
 {A Note on Doubly Nonlinear Parabolic \\
 Systems with Unilateral Constraint}

\author[Michal Bene\v{s}]{Michal Bene\v{s}}
\address{%
Czech Technical University in Prague\\
Th\'{a}kurova 7\\
166 29 Prague 6\\
Czech Republic}

\email{benes@mat.fsv.cvut.cz}


\subjclass{Primary 35K40; Secondary 35K85}

\keywords{Second-order parabolic systems; Unilateral problems and
variational inequalities for parabolic PDE}

\date{January 1, 2004}

\begin{abstract}
We prove the existence and uniqueness of the solution to the doubly
nonlinear parabolic systems with mixed boundary conditions. Due to
the unilateral constraint the problem comes as a variational
inequality. We apply the penalty method and Gronwall's technique to
prove the existence and uniqueness of the variational solution.
\end{abstract}

\maketitle

\section{Introduction}
Let $\Omega$ be a bounded domain in $\mathbb{R}^N$, $N=1,2$ or $3$,
with a smooth boundary $\partial\Omega$ for $N=2$ or $N=3$. Let
$\Gamma_1$, $\Gamma_2$ and $\Gamma_3$ be open disjoint subsets of
$\Gamma=\pom$ (not necessarily connected) such that $\Gamma =
\overline{\Gamma_1}\cup\overline{\Gamma_2}\cup\overline{\Gamma_3}$
and $\textmd{meas}_{N-1}(\Gamma_i)>0$ for $i=1,2,3$. For a positive
$T$ we denote $Q_T=\Omega\times(0,T)$,
$S_T=\partial\Omega\times(0,T)$. $T$ is supposed to be fixed
throughout the paper. We study the following system ($j=1,\dots,m$)
\begin{align}
\partial_t B^j(\bfu)-\nabla\cdot\left(K^{ji}(\bfu)
\nabla u^i+\bfe^j(\bfu)\right) &=
F^j(\bfx,t,\bfu) &&\textmd{ in }\; Q_T,
\label{eq1}
\\
\bfu(\bfx,0)&=\bfu_0(\bfx) &&\textmd{ in }\Omega,
\label{eq2}
\\
\bfu&={\bf0}  &&\textmd{ on } \Gamma_1 \times (0,T),\label{eq3}
\\
\left(K^{ji}(\bfu)\nabla
u^i+\bfe^j(\bfu)\right)\cdot\bfn&=g^j(\bfx,t,\bfu)&&\textmd{ on }
\Gamma_2 \times (0,T), \label{eq4}
\end{align}
\begin{equation}\label{eq5}
\left. \; \quad \quad
\begin{array}{c}
\begin{tabular}{rl}
$u^j \leq 0$&
\\
$\left(K^{ji}(\bfu) \nabla u^i+\bfe^j(\bfu)\right)\cdot\bfn\leq 0 $&
\\
$u^j \left[\left(K^{ji}(\bfu) \nabla
u^i+\bfe^j(\bfu)\right)\cdot\bfn \right]=0$
\\
\end{tabular}
\end{array}\right\}
\qquad  \textmd{  on  }  \Gamma_3\times (0,T).
\end{equation}
In  \eqref{eq1}--\eqref{eq5}, $\bfn$ denotes the outer unit normal
to $\partial\Omega$, $\bfu=(u^1,\dots,u^m)$ represents the unknown
fields of state variables, the vector $\bfu_0=(u^1_0,\dots,u^m_0)$
describes the initial condition. By $\bfB$, $\bfK^j$
($j=1,\dots,m$), $\bfe^j$, $\bfF$, $\bfg$, we denote the vectors
$\bfB=(B^1,\dots,B^m)$, $\bfK^j=(K^{j1},\dots,K^{jm})$,
$\bfe^j=(e^{j}_1,\dots,e^{j}_N)$, $\bfF=(F^1,\dots,F^m)$,
$\bfg=(g^1,\dots,g^m)$, which are smooth functions of primary
unknowns $\bfu$. Hence, the problem is strongly nonlinear.

Systems of equations like \eqref{eq1}--\eqref{eq5} arise in a
variety of physical situations. For example, they describe the
evolution of the dual water flow through porous media (cf.
\cite{gerkegenuchten}) and, for instance, heat and moisture transfer
in porous structures (see \cite{KuKi1997}).

A considerable effort has been invested into qualitative properties
of scalar problems with $m=1$  (cf.
\cite{Diaz1996,Eden1990,Eden1994,ivanov,zeman1991}). However, much
less attention has been given to the qualitative properties of
systems for doubly nonlinear equations of type \eqref{eq1}. The
global existence of weak solutions to \eqref{eq1}--\eqref{eq2} in
bounded domains subject to mixed Dirichlet-Neumann boundary
conditions has been shown by Alt \& Luckhaus in
\cite{AltLuckhaus1983} assuming the function $B^j$ to be monotone
and $\bfg\equiv{\bf0}$. This result has been extended in various
different directions
\cite{FiloKacur1995,Kacur1997,Kacur1990a,Kacur1990b}. For instance,
Filo \& Ka\v{c}ur~\cite{FiloKacur1995} proved the local existence of
the weak solution for the system with nonlinear Neumann boundary
conditions and under more general growth conditions on
nonlinearities in $\bfu$. The uniqueness of the solution has been
proven in \cite{AltLuckhaus1983} under the additional assumption
$\partial_t B^j(\bfu)\in L^1$ and assuming the elliptic term in the
form $\left(K^{ji}(\bfx) \nabla u^i+\bfe^j(\bfu)\right)$. In
\cite{hachimi2001}, El Ouardi \& El Hachimi proved the existence of
the regular attractor for Dirichlet problem to nonlinear parabolic
systems with Laplacian in the elliptic part of the problem. In
\cite{hachimi2006}, the same authors proved the existence of
solutions for doubly nonlinear systems including the
\emph{p}-Laplacian
 as the principal part of the operator considering the
Dirichlet boundary conditions on the whole part of the domain.

In the present paper we prove the existence and uniqueness of the
variational solution to the doubly nonlinear parabolic system
\eqref{eq1}--\eqref{eq4} including the unilateral constraint
\eqref{eq5}. We adapt ideas presented by Filo \&
Ka\v{c}ur~\cite{FiloKacur1995} to extend their results to
variational inequalities. This paper is organized as follows. In
Sections \ref{notations}, \ref{Structure and data properties} and
\ref{Auxiliary_results}, we introduce basic notations, specify
structure conditions and assumptions on data in the problem and
recall some important auxiliary results needed below. In Section
\ref{variational solution}, we formulate our problem as the
variational inequality and reformulate the solved problem in the
operator form in appropriate function spaces. The main results, the
existence and uniqueness of the variational solution, are proved in
Sections \ref{sec:main_result_existence} and
\ref{sec:main_result_uniqueness} via the penalty method and
Gronwall's technique.

\section{Preliminaries}
\subsection{Notations}
\label{notations}
Vectors, vector functions and operators acting on vector functions
are denoted by~boldface letters. Unless specified otherwise, we use
Einstein summation convention for indices running from $1$ to $m$.
Throughout the paper, we will always use positive constants $C$,
$c$, $c_1$, $c_2$, $\dots$, which are not specified and which may
differ from line to line.

For an arbitrary $p\in [1,+\infty]$, $L^p(\Omega)$ denotes the usual
Lebesgue space equipped with the norm $\|\cdot\|_{L^p(\Omega)}$, and
$W^{k,p}(\Omega)$, $k\geq 0$ ($k$ need not to be an integer, see
\cite{KufFucJoh1977}),  denotes the usual Sobolev space with the
norm $\|\cdot\|_{W^{k,p}(\Omega)}$. Let
\begin{displaymath}
{E}:=\left\{\bfu\in C^\infty(\overline{\Omega})^m;\;
 \, {\textmd{supp}\, \bfu}  \cap \Gamma_1 =
\emptyset \right\}
\end{displaymath}
and $\mathbb{V}$ be a closure of $E$ in the norm of
$W^{1,2}(\om)^m$. By $\langle \cdot,\cdot \rangle$ we denote the
duality between $\mathbb{V}$ and $\mathbb{V}^*$.

\subsection{Structure and data properties}
\label{Structure and data properties}

Next we introduce our assumptions on the functions in
\eqref{eq1}--\eqref{eq5}:
\begin{itemize}

\item[(A1)] there is a strictly convex $C^1$-function
$\Phi:\mathbb{R}^m\rightarrow \mathbb{R}$, $\Phi({\bf0})=0$,
$\nabla\Phi({\bf0})={\bf0}$, such that
\begin{equation}\label{conH1}
\bfB(\bfz)=\nabla\Phi(\bfz) \quad \forall  \bfz \in \mathbb{R}^m.
\end{equation}
The Legendre transform $\Psi(\bfz):=
\int_{0}^{1}(\bfB(\bfz)-\bfB(s\bfz))\cdot \bfz \, {\rm d}s$
satisfies
\begin{equation}\label{conH2}
\Psi(\bfz)\geq c_1|\bfz|^{\nu+1}-c_2 \quad (\nu>0)\;  \forall
\bfz\in \mathbb{R}^m;
\end{equation}

\item[(A2)] $K^{ji}:\mathbb{R}^m\rightarrow \mathbb{R}$ and
$\bfe^j:\mathbb{R}^m\rightarrow \mathbb{R}^N$ are continuous and
($i,j=1,\dots,m$ and $k=1,\dots,N$)
\begin{equation}\label{con_9a}
| K^{ji}(\bfz) | + | e_k^j(\bfz) | \leq c  \quad \forall  \bfz \in
\mathbb{R}^m.
\end{equation}
$(K^{ji})$ is a positive-definite matrix satisfying
\begin{displaymath}\label{con_9b}
K^{ji}(\bfz)\xi^i\xi^j \geq c|\bfxi|^2 \quad \forall \bfxi, \bfz \in
\mathbb{R}^m;
\end{displaymath}

\item[(A3)] the functions $\bfF:Q_T\times\mathbb{R}^m\rightarrow\mathbb{R}^m$ and
$\bfg:\Gamma_2 \times (0,T)\times\mathbb{R}^m\rightarrow
\mathbb{R}^m$ are continuous and
\begin{equation}\label{con_11}
\left\{
\begin{array}{rl}
|\bfF(\bfx,t,\bfz)|\leq c(|\bfz|^p+1),& \forall \bfz\in
\mathbb{R}^m,\; [\bfx,t]\in Q_T,
\\
|\bfg(\bfx,t,\bfz)|\leq c(|\bfz|^{\alpha}+1),& \forall \bfz\in
\mathbb{R}^m,\; [\bfx,t]\in \Gamma_2 \times (0,T);
\end{array}\right.
\end{equation}

\item[(A4)] assume $p\leq \nu$ and that either one of the following
conditions is satisfied:

\noindent (i)  $0<\alpha \leq \min\left\{ \nu,1 \right\}$

\noindent (ii) $1<\alpha<(N+\alpha+1)/N$ and

\begin{displaymath}
\alpha < \left\{
\begin{array}{ll}
{(\nu+1)}/{2} & \textmd{for } N=1 , \\
{(3\nu+1)}/{(3+\nu)} & \textmd{for }N=2, \\
\nu+2-\sqrt{\nu^2-\nu+3} & \textmd{for }N=3; \\
\end{array} \right.
\end{displaymath}

\item[(A5)] for initial data we assume $\bfu_0\in
W^{1,2}(\Omega)^m$ and $\bfu_0\cdot\bfB(\bfu_0)\in L^1(\Omega)$.

\end{itemize}


\subsection{Auxiliary results}
\label{Auxiliary_results}

Due to the trace theorem \cite{KufFucJoh1977} the trace mapping
$\mathfrak{T}:W^{1,2}(\om)\rightarrow L^q(\partial\Omega)$, $q\geq
1$ for $N=1,2$ and $1\leq q \leq 4$ for $N=3$, is continuous, i.e.
there exists a constant $c_{tr}$ such that
\begin{equation}\label{trace_th}
\|v\|_{L^q(\partial\Omega)} \leq c_{tr}\|v\|_{W^{1,2}(\om)} \textmd{
for all } v\in W^{1,2}(\om).
\end{equation}
Let (A4) be satisfied. Then (see \cite[Corollary 2]{FiloKacur1995})
\begin{equation}\label{interpol_ineq_gen}
 \int_{\partial\om} |v|^{\alpha+1} {\rm d}\Gamma
 \leq \eta   \|v\|^2_{W^{1,2}(\om)}
  + C(\eta)\int_{\om} |v|^{\nu+1} {\rm d}\bfx
 \; \textmd{ for all }  v\in W^{1,2}(\om).
\end{equation}

The following assertion is proved in \cite[Lemma 2 and
3]{FiloKacur1995}: let $\left\{ \bfw_k
\right\}_{k=1}^{\infty}\subset L^2(0,T;\mathbb{V})\cap
L^{\infty}(Q_T)^m$ and
\begin{displaymath}
\|\bfw_k\|_{L^2(0,T;\mathbb{V})}+\|\bfw_k\|_{L^{\infty}(Q_T)^m}<C,
\; k=1,2,\dots
\end{displaymath}
Moreover, let $\bfw_k\rightarrow \bfw$ a.e. on $Q_T$. Then
\begin{equation}\label{strong_conv_th}
\left\{
\begin{array}{rll}
\bfw_k \rightarrow \bfw & \;{\rm in}\; L^{q+1}(Q_T)^m, & 0\leq q
<p^*,
\\
\bfw_k \rightarrow \bfw& \;{\rm in}\; L^{s+1}(S_T)^m, &
0<s<(N+\min\left\{s,\nu\right\}+1)/N.
\end{array}\right.
\end{equation}

\section{The variational solution, existence and uniqueness}

\subsection{Variational solution}\label{variational solution}

Let us define the closed and convex set
\begin{equation}\label{konvex_K}
\mathcal{K}:=\left\{ \bfv \in \mathbb{V}; \;  v^j \leq 0 \textmd{
a.e. on }
 \Gamma_3, \; j=1,\dots,m \right\}.
\end{equation}

\begin{defn}
A vector function $\bfu\in L^2(0,T;\mathcal{K})\cap
L^{\infty}(0,T;L^{\nu+1}(\Omega)^m)$ with ${\partial_t\bfB(\bfu)}\in
L^2(0,T;\mathbb{V}^*)$, $\bfB(\bfu)\in L^1(Q_T)^m$, is a variational
solution to the system \eqref{eq1}--\eqref{eq5} iff

\begin{itemize}

\item[(i)]
\begin{multline}\label{eq41}
\int_0^T\langle{\partial_t \bfB(\bfu)},\bfvarphi-\bfu\rangle{\rm d}t
+ \int_{Q_T}\left( K^{ji}(\bfu)\nabla u^i+\bfe^j(\bfu)\right)
\cdot\nabla(\varphi^j-u^j) \;{\rm d}Q_T
\\
\geq
\int_0^T\!\!\!\int_{\Gamma_2}\bfg(\bfx,t,\bfu)\cdot(\bfvarphi-\bfu)
\;{\rm d}S_T+\int_{Q_T}\bfF(\bfx,t,\bfu)\cdot(\bfvarphi-\bfu)\;{\rm
d}Q_T
\end{multline}
holds for all $\bfvarphi \in L^2(0,T;\mathcal{K})\cap
L^{\infty}(Q_T)^m$ and $\bfu(0)=\bfu_0$ in $\Omega$;

\item[(ii)]
\begin{displaymath}
\int_0^T\langle{\partial_t \bfB(\bfu)},\bfv\rangle{\rm d}t= -
\int_{Q_T}(\bfB(\bfu)-\bfB(\bfu_0))\cdot\partial_t \bfv \;{\rm d}Q_T
\end{displaymath}
for all $\bfv \in L^{2}(0,T;\mathbb{V})\cap L^{\infty}(Q_T)^m$ with
$\partial_t\bfv\in L^{\infty}(Q_T)^m$, $\bfv(T)={\bf0}$.

\end{itemize}
\end{defn}

\begin{defn}
Define an operator $\mathscr{T}$,
\begin{displaymath}
\mathscr{T} : \left\{ \bfpsi; \, \bfpsi \in L^2(0,T;\mathbb{V}),
{\partial_t\bfB(\bfpsi)} \in L^2(0,T;\mathbb{V}^*) \right\}
\rightarrow L^2(0,T;\mathbb{V}^*),
\end{displaymath}
given by the equation
\begin{multline}\label{operator_T}
\int_{Q_T}\!\!\langle\mathscr{T}(\bfpsi),\bfv \rangle \;{\rm d}t =
\int_0^T\!\!\langle{\partial_t \bfB(\bfpsi)},\bfv\rangle{\rm d}t +
\int_{Q_T}\!\!\left( K^{ji}(\bfpsi)\nabla
\psi^i+\bfe^j(\bfpsi)\right) \cdot\nabla v^j \;{\rm d}Q_T
\\
-\int_0^T\!\!\!\int_{\Gamma_2}\bfg(\bfx,t,\bfpsi)\cdot \bfv \;{\rm
d}S_T - \int_{Q_{T}} \bfF(\bfx,t,\bfpsi)\cdot\bfv \;{\rm d}Q_T
\end{multline}
for all $\bfv \in L^2(0,T;\mathbb{V})\cap L^{\infty}(Q_T)^m$.
\end{defn}

\begin{rem}
If $\bfu$ is the variational solution to the system
\eqref{eq1}--\eqref{eq5} then the inequality \eqref{eq41} can be
replaced by
\begin{equation}
\int_{Q_T}\!\!\langle\mathscr{T}(\bfu),\bfvarphi-\bfu \rangle \;{\rm
d}t \geq 0
\end{equation}
for all $\bfvarphi \in L^2(0,T;\mathcal{K})\cap L^{\infty}(Q_T)^m$.
\end{rem}

\subsection{The existence of the solution}
\label{sec:main_result_existence}

\begin{thm}\label{main_result_existence}
Let the assumptions {\rm (A1)--(A5)} be satisfied. Then there exists
the variational solution to \eqref{eq1}--\eqref{eq5}.
\end{thm}

\begin{defn}\label{penalty_operator}
Let $\mathcal{S} \neq \emptyset$ be a closed and convex subset of a
reflexive Banach space $Y$. An operator $\mathscr{P}:Y\rightarrow
Y^*$ is called a penalty operator associated with
$\mathcal{S}\subset Y$ if
$$
\mathscr{P}(\zeta)=0_{Y^*} \Leftrightarrow \zeta \in \mathcal{S}.
$$
\end{defn}

Theorem \ref{main_result_existence} is a consequence of the
following
\begin{thm}\label{main_existence_theorem}
Let $\mathcal{K}$ be the closed and convex subset of the space
$\mathbb{V}$ defined by \eqref{konvex_K} and $\mathscr{T}$ be the
operator given by the equation \eqref{operator_T}. Let the
assumptions {\rm (A1)--(A5)} be satisfied. Then

 {\rm \textbf{(1)}} the operator
$\beta:\mathbb{V}\rightarrow \mathbb{V}^*$ given by the equation
\begin{equation}
\langle \beta(\bfpsi),\bfv \rangle=\int_{\Gamma_3} \bfpsi^+ \cdot
\bfv \;{\rm d}\Gamma \; \textmd{ for all } \bfv\in \mathbb{V}, \;
(\psi^+)^j=\max\left\{\psi^j(\bfx),0 \right\},
\end{equation}
represents a penalty operator associated with $\mathcal{K}$.

 {\rm \textbf{(2)}} For all $\varepsilon>0$ there exists
 $\bfu_{\varepsilon} \in  L^2(0,T;\mathbb{\mathbb{V}})$ with
${\partial_t\bfB(\bfu_{\varepsilon})} \in L^2(0,T;\mathbb{V}^*)$
(the variational solution of the penalized problem
$(P_{\varepsilon})$) such that
\begin{equation}\label{penalized_problem_weak}
\int_0^T \langle \mathscr{T}(\bfu_{\varepsilon}),\bfvarphi
\rangle{\rm d}t+\frac{1}{\varepsilon}\int_0^T \langle
\beta(\bfu_{\varepsilon}),\bfvarphi \rangle {\rm d}t =0
\end{equation}
for all $\bfvarphi\in L^2(0,T;\mathbb{V})\cap L^{\infty}(Q_T)^m$ and
$\bfu_{\varepsilon}(0)=\bfu^0$ in $\Omega$.

 {\rm \textbf{(3)}} Let $\varepsilon_n\rightarrow 0^+$
 as $n\rightarrow \infty$. The
sequence $\bfu_{\varepsilon_n}$ of solutions to Problems
$(P_{\varepsilon_n})$ converges weakly in $L^2(0,T;\mathbb{V})$
toward the variational solution $\bfu$ of \eqref{eq1}--\eqref{eq5}.
\end{thm}
\begin{proof}
Part (1) \quad Due to \eqref{trace_th} the penalty operator $\beta$
is well defined and the equivalence $\beta(\bfu)={\bf0}$ iff $\bfu
\in \mathcal{K}$ is straightforward.

Part (2) \quad The assertion follows from \cite[Theorem 1 and Remark
1]{FiloKacur1995}.

Part (3) \quad Test \eqref{penalized_problem_weak} by
$\bfvarphi=\bfu_{\varepsilon}\chi_{(0,t)}$ (here $\chi_{(0,t)}$
denotes the characteristic function of $(0,t)$) to get
\begin{equation}\label{penalized_problem_weak_test}
\int_0^T \langle
\mathscr{T}(\bfu_{\varepsilon}),\bfu_{\varepsilon}\chi_{(0,t)}
\rangle{\rm d}s+\frac{1}{\varepsilon}\int_0^T \langle
\beta(\bfu_{\varepsilon}),\bfu_{\varepsilon}\chi_{(0,t)} \rangle
{\rm d}s =0
\end{equation}
and consequently
\begin{multline}\label{operator_T_hat}
\int_0^t \!\! \langle {\partial_s\bfB(\bfu_{\varepsilon})},
\bfu_{\varepsilon}\rangle{\rm d}s  + \!\! \int_{Q_t}\!\!\!\!\!
\left( K^{ji}(\bfu_{\varepsilon})\nabla u^i_{\varepsilon}
  +\bfe^j(\bfu_{\varepsilon})\right)\cdot\nabla u^j_{\varepsilon}{\rm
  d}Q_t
+\frac{1}{\varepsilon}\int_0^t \!\!\!\! \langle
\beta(\bfu_{\varepsilon}),\bfu_{\varepsilon}\rangle{\rm d}s
\\
=\int_0^t\!\!\!\int_{\Gamma_2}\bfg(\bfx,s,\bfu_{\varepsilon})\cdot
\bfu_{\varepsilon} \;{\rm d}S_t+\int_{Q_{t}}
\bfF(\bfx,s,\bfu_{\varepsilon})\cdot \bfu_{\varepsilon} \;{\rm
d}Q_t.
\end{multline}
Integrating by parts in the parabolic term, \eqref{operator_T_hat}
yields
\begin{multline}\label{eq201}
\int_{\Omega}\!\!\Psi(\bfu_{\varepsilon}(t)) \;{\rm d}\bfx + \!\!
\int_{Q_t}\!\! \!\!\! \left( K^{ji}(\bfu_{\varepsilon})\nabla
u^i_{\varepsilon} +\bfe^j(\bfu_{\varepsilon})\right)\cdot\nabla
u^j_{\varepsilon}{\rm d}Q_t +\frac{1}{\varepsilon}\int_0^t
\!\!\!\langle
\beta(\bfu_{\varepsilon}),\bfu_{\varepsilon}\rangle{\rm d}s
\\
=\int_{\Omega}\Psi(\bfu(0))\;{\rm d}\bfx
+\int_0^t\!\!\!\int_{\Gamma_2}\bfg(\bfx,s,\bfu_{\varepsilon})\cdot
\bfu_{\varepsilon} \;{\rm d}S_t+\int_{Q_{t}}
\bfF(\bfx,s,\bfu_{\varepsilon})\cdot \bfu_{\varepsilon} \;{\rm
d}Q_t.
\end{multline}
Now, taking into account (A1) together with (A3), one obtains
\begin{multline}\label{eq201b}
\int_{\Omega}\!\!\Psi(\bfu_{\varepsilon}(t)) \;{\rm d}\bfx + \!\!
\int_{Q_t}\!\! \!\!\! \left( K^{ji}(\bfu_{\varepsilon})\nabla
u^i_{\varepsilon} +\bfe^j(\bfu_{\varepsilon})\right)\cdot\nabla
u^j_{\varepsilon}{\rm d}Q_t +\frac{1}{\varepsilon}\int_0^t
\!\!\!\langle
\beta(\bfu_{\varepsilon}),\bfu_{\varepsilon}\rangle{\rm d}s
\\
\leq c_1\int_{\Omega}\Psi(\bfu(0)) \;{\rm d}\bfx
 + c_2 \int_0^t\!\!\!\int_{\Gamma_2}|\bfu_{\varepsilon}|^{\alpha+1}
 \;{\rm d}S_t + c_3 \int_{Q_{t}}
|\bfu_{\varepsilon}|^{p+1} \,{\rm d}Q_t .
\end{multline}
Further, \eqref{eq201b}, interpolation inequality
\eqref{interpol_ineq_gen} and (A1)--(A4) yield
\begin{multline}\label{eq201c}
\int_{\Omega}\Psi(\bfu_{\varepsilon}(t)) \;{\rm d}\bfx +\int_0^t
c_1\|\bfu_{\varepsilon}(s)\|^2_{\mathbb{V}} \;{\rm d}s
+\frac{1}{\varepsilon}\int_0^t \langle
\beta(\bfu_{\varepsilon}),\bfu_{\varepsilon}\rangle{\rm d}s
\\
\leq c_2\int_{\Omega}\Psi(\bfu(0)) \;{\rm d}\bfx + c_3 \int_{Q_{t}}
\Psi(\bfu_{\varepsilon}(s)) \;{\rm d}Q_t  + c_4.
\end{multline}
Applying the Gronwall's inequality to \eqref{eq201c} we arrive at
\begin{equation}\label{eq201d}
\int_{\Omega}\Psi(\bfu_{\varepsilon}(t)) \;{\rm d}\bfx \leq
\left(c_1\int_{\Omega}\Psi(\bfu(0)) \;{\rm d}\bfx + c_2\right)
\left( 1+ c_3 t \exp(c_3 t)\right)
\end{equation}
for a.e. $0\leq t\leq T$. Further, \eqref{eq201c}--\eqref{eq201d}
imply
\begin{equation}\label{eq203}
\int_0^t |\langle\beta(\bfu_{\varepsilon}),\bfu_{\varepsilon}
\rangle |\, {\rm d}s \leq \varepsilon c ,
\end{equation}
where $c$ is independent of $\varepsilon$. Hence, as $\varepsilon
\rightarrow 0$ we have
\begin{equation*}
\int_0^T \langle\beta(\bfu_{\varepsilon}),\bfu_{\varepsilon} \rangle
\, {\rm d}s \rightarrow 0 .
\end{equation*}
Analogously,
\begin{equation*}
\int_0^T \langle\beta(\bfu_{\varepsilon}),\bfv \rangle \, {\rm d}s
\rightarrow 0
\end{equation*}
for every $\bfv \in L^2(0,T;\mathbb{V})$ and
$\beta(\bfu_{\varepsilon})\rightarrow 0$. Further \eqref{eq201c} and
\eqref{eq201d} imply
\begin{equation}\label{est100}
\sup_{0\leq t \leq T}\int_{\Omega}\Psi(\bfu_{\varepsilon}(t)) \;{\rm
d}\bfx +\int_0^T  \|\bfu_{\varepsilon}(t)\|^2_{\mathbb{V}} \;{\rm
d}t \leq c,
\end{equation}
which yields (by (A1))
\begin{equation}\label{est100b}
\sup_{0\leq t \leq T}\int_{\Omega} |\bfu_{\varepsilon}(t)|^{\nu+1}
\;{\rm d}\bfx +\int_0^T  \|\bfu_{\varepsilon}(t)\|^2_{\mathbb{V}}
\;{\rm d}t \leq c.
\end{equation}
Since any bounded set in a reflexive Banach space is weakly
sequentially compact, we can find a subsequence $\left\{
\bfu_{\varepsilon_{n}} \right\}$ such that $\bfu_{\varepsilon_{n}}
\rightharpoonup \bfu\in L^2(0,T;\mathbb{V})$. Let $\bfv \in
L^2(0,T;\mathbb{V})$ be arbitrary fixed. Then
\begin{equation}
\int_0^T \langle\beta(\bfv)
-\beta(\bfu_{\varepsilon_{n}}),\bfv-\bfu_{\varepsilon_{n}} \rangle
\, {\rm d}t \geq 0
\end{equation}
yields
\begin{equation}
\int_0^T \langle\beta(\bfv),\bfv-\bfu \rangle \, {\rm d}t \geq 0.
\end{equation}
Choose $\bfv=\bfu+a\bfz$ ($a>0$, $\bfz \in L^2(0,T;\mathbb{V})$
arbitrary), hence
\begin{equation}
\int_0^T \langle\beta(\bfu+a \bfz),\bfz \rangle \, {\rm d}t \geq 0
\end{equation}
and letting $a \rightarrow 0^+$ we have
\begin{equation}
\int_0^T \langle\beta(\bfu),\bfz \rangle \, {\rm d}t \geq 0
\end{equation}
for every $\bfz \in L^2(0,T;\mathbb{V})$. Hence $\beta(\bfu)=0$,
that is, $\bfu \in
 L^2(0,T;\mathcal{K})$.
For arbitrary fixed $\bfv \in  L^2(0,T;\mathcal{K})$, we deduce
using the equation \eqref{penalized_problem_weak}
\begin{equation}\label{ineq_10}
\int_0^T \langle
{\mathscr{T}}(\bfu_{\varepsilon_n}),\bfv-\bfu_{\varepsilon_n}
\rangle{\rm d}t =\frac{1}{\varepsilon_n}\int_0^T \langle
\beta(\bfv)-\beta(\bfu_{\varepsilon_n}), \bfv-\bfu_{\varepsilon_n}
\rangle {\rm d}t \geq 0.
\end{equation}
In the rest of this section we prove that as $\bfu_{\varepsilon_{n}}
\rightharpoonup \bfu\in L^2(0,T;\mathcal{K})$ then \eqref{ineq_10}
reads
\begin{displaymath}
\int_0^T\langle {\mathscr{T}}(\bfu),\bfvarphi-\bfu \rangle \;{\rm
d}t \geq 0 \; \textmd{ for every } \bfvarphi\in
L^{2}(0,T;\mathcal{K}).
\end{displaymath}
In order to do that we prove the following
\begin{lem}\label{weak_convergence}
The sequence $\bfu_{\varepsilon_n}$ satisfies
\begin{equation}\label{weakconv_a}
\left\{
\begin{array}{rclll}
{\partial_t\bfB(\bfu_{\varepsilon_n})} \cdot \bfv  &\rightarrow&
{\partial_t\bfB(\bfu)}\cdot \bfv &{\rm in }&  L^1(Q_T),
\\
\bfF(\bfx,t,\bfu_{\varepsilon_n}) \cdot \bfv  & \rightarrow &
\bfF(\bfx,t,\bfu) \cdot \bfv  &{\rm in }&  L^1(Q_T),
\\
\bfg(\bfx,t,\bfu_{\varepsilon_n}) \cdot \bfv  & \rightarrow &
\bfg(\bfx,t,\bfu) \cdot \bfv  &{\rm in }& L^1((0,T)\times \Gamma_2)
\end{array}
\right.
\end{equation}
for every $\bfv \in L^2(0,T;\mathbb{V})\cap L^{\infty}(Q_T)^m$ and
\begin{equation}\label{weakconv_b}
\left\{
\begin{array}{rclll}
K^{ji}(\bfu_{\varepsilon_n})\nabla
u^i_{\varepsilon_n}&\rightharpoonup & K^{ji}(\bfu)\nabla u^i  & {\rm
in } & L^2(Q_T)^N,
\\
\bfe^j(\bfu_{\varepsilon_n}) & \rightharpoonup & \bfe^j(\bfu) &{\rm
in }& L^2(Q_T)^N.
\end{array}
\right.
\end{equation}
\end{lem}
\begin{proof}
Due to (A1)--(A4), \eqref{operator_T} and
\eqref{penalized_problem_weak} we have
\begin{equation}
\sup_{\|\bfv\|_{L^2(0,T;\mathbb{V})}\leq 1} \Bigg| \int_{Q_T}
{\partial_t \bfB(\bfu_{\varepsilon_n})} \cdot \bfv \;{\rm d}Q_T
\Bigg| \leq c.
\end{equation}
Hence, the sequence $\left\{ {\partial_t\bfB(\bfu_{\varepsilon_n})}
\right\}$ is uniformly bounded in $L^2(0,T;\mathbb{V}^*)$ and,
consequently, there exists a subsequence and $\bfchi$ such that
\begin{equation}\label{eq550}
{\partial_t\bfB(\bfu_{\varepsilon_n})} \rightharpoonup \bfchi
\textmd{ in } L^2(0,T;\mathbb{V}^*).
\end{equation}
The identity
\begin{equation}\label{eq551}
\int_{Q_T} {\partial_t\bfB(\bfu_{\varepsilon_n})}\cdot \bfv \;{\rm
d}Q_T = -\int_{Q_T}
\left(\bfB(\bfu_{\varepsilon_n})-\bfB(\bfu^0)\right) \cdot \frac{d
\bfv}{dt} \; {\rm d}Q_T
\end{equation}
holds for every $\bfv \in L^2(0,T;\mathbb{V})\cap
L^{\infty}(Q_T)^m$, ${d \bfv}/{dt}\in L^{\infty}(Q_T)^m$. Using the
compactness argument one can show in the same way as in \cite[Lemma
1.9]{AltLuckhaus1983} the convergence
\begin{equation}\label{eq555}
\bfB(\bfu_{\varepsilon_n})\rightarrow \bfB(\bfu)\textmd{ in
}L^1(Q_T).
\end{equation}
Taking the limit in \eqref{eq551} and using \eqref{eq550} and
\eqref{eq555} we get
\begin{equation}\label{eq552}
\int_{Q_T} \bfchi \cdot \bfv \;{\rm d}Q_T =  -\int_{Q_T}
\left(\bfB(\bfu)-\bfB(\bfu^0)\right) \cdot \frac{d \bfv}{dt} \; {\rm
d}Q_T
\end{equation}
for every $\bfv \in L^2(0,T;\mathbb{V})\cap L^{\infty}(Q_T)^m$, ${d
\bfv}/{dt}\in L^{\infty}(Q_T)^2$. Now \eqref{eq552} yields
\begin{equation}\label{eq553}
\int_{Q_T} \bfchi \cdot \bfv \;{\rm d}Q_T = \int_{Q_T}
\partial_t\bfB(\bfu) \cdot  \bfv  \; {\rm d}Q_T
\end{equation}
for every $\bfv \in L^2(0,T;\mathbb{V})\cap L^{\infty}(Q_T)^m$ and
therefore $\bfchi =
\partial_t\bfB(\bfu)$.

Since $B^j$ is strictly monotone and from \eqref{eq555} it follows
that \cite[Proposition 3.35]{Kacur1990a}
\begin{equation}\label{conv101}
\bfu_{\varepsilon_n} \rightarrow \bfu \textmd{ a.e. in } Q_T.
\end{equation}
Hence we have
\begin{equation}\label{conv99}
\left\{
\begin{array}{rcll}
\bfF(\bfx,t,\bfu_{\varepsilon_n}) & \rightarrow & \bfF(\bfx,t,\bfu)
&{\textmd{ a.e. in }} Q_T,
\\
\bfg(\bfx,t,\bfu_{\varepsilon_n}) & \rightarrow & \bfg(\bfx,t,\bfu)
&  {\textmd{ a.e. in }} (0,T)\times\Gamma_2.
\end{array} \right.
\end{equation}
Now \eqref{con_11} and \eqref{conv99} imply that for every  $\bfv
\in L^2(0,T;\mathbb{V})\cap L^{\infty}(Q_T)^m$ one obtains
\begin{equation*}
\left\{
\begin{array}{rclll}
\bfF(\bfx,t,\bfu_{\varepsilon_n}) \cdot \bfv  & \rightarrow &
\bfF(\bfx,t,\bfu) \cdot \bfv  &{\rm in }&  L^1(Q_T),
\\
\bfg(\bfx,t,\bfu_{\varepsilon_n}) \cdot \bfv  & \rightarrow &
\bfg(\bfx,t,\bfu) \cdot \bfv  &{\rm in }& L^1((0,T)\times \Gamma_2).
\end{array}
\right.
\end{equation*}
Further, \eqref{strong_conv_th} yields the convergence
\begin{equation*}
\left\{
\begin{array}{rclll}
\bfF(\bfx,t,\bfu_{\varepsilon_n}) \cdot \bfu_{\varepsilon_n}  &
\rightarrow & \bfF(\bfx,t,\bfu) \cdot \bfu  &{\rm in }&  L^1(Q_T),
\\
\bfg(\bfx,t,\bfu_{\varepsilon_n}) \cdot \bfu_{\varepsilon_n}  &
\rightarrow & \bfg(\bfx,t,\bfu) \cdot \bfu  &{\rm in }&
L^1((0,T)\times \Gamma_2).
\end{array}
\right.
\end{equation*}
Now \eqref{con_9a} and \eqref{conv101} give the convergence
$\bfe^j(\bfu_{\varepsilon_n})\rightharpoonup \bfe^j(\bfu)$ in
$L^2(Q_T)^m$.

Using \eqref{con_9a} and \eqref{est100} we arrive at
\begin{equation}\label{est120}
\|K^{ji}(u^i_{\varepsilon})\nabla u^i_{\varepsilon}\|_{L^2(Q_T)^N}
\leq  C.
\end{equation}
Hence there exists $\bfvarphi^j\in L^2(Q_T)^N$ such that
\begin{equation}\label{conv105}
K^{ji}(\bfu_{\varepsilon_n})\nabla
u^i_{\varepsilon_n}\rightharpoonup \bfvarphi^j \textmd{ in }
L^2(Q_T)^N.
\end{equation}
To prove $\bfvarphi^j = K^{ji}(\bfu)\nabla u^i$ we follow the trick
of Minty-Browder in reflexive spaces. Obviously, for every $\bfw\in
L^2(0,T;\mathbb{V})$ we have
\begin{equation}\label{mon_100}
\int_{Q_T} \left( K^{ji}(\bfu_{\varepsilon_n})\nabla
u^i_{\varepsilon_n}-K^{ji}(\bfu_{\varepsilon_n})\nabla w^i
\right)\cdot \left( \nabla u^j_{\varepsilon_n}-\nabla w^j \right)
\;{\rm d}Q_T \geq 0.
\end{equation}
Letting $\varepsilon_n \rightarrow 0$ one obtains
\begin{equation}
\int_{Q_T} \left( \bfvarphi^j - K^{ji}(\bfu)\nabla w^i\right) \cdot
\left(\nabla u^j - \nabla w^j \right)\;{\rm d}Q_T \geq 0.
\end{equation}
Fix any $\bfv\in L^2(0,T;\mathbb{V})$ and set $\bfw=\bfu-\tau \bfv$
$(\tau>0)$ to obtain (letting $\tau\rightarrow 0$)
\begin{equation}\label{ineq_105}
\int_{Q_T} \left( \bfvarphi^j - K^{ji}(\bfu)\nabla u^i \right) \cdot
\nabla v^j \;{\rm d}Q_T \geq 0.
\end{equation}
Replacing $\bfv$ by $-\bfv$ we deduce that equality holds above.
Hence we get
\begin{equation}\label{conv110}
\bfvarphi^j = K^{ji}(\bfu)\nabla u^i  \quad {\textmd{ a.e. in }} Q_T
.
\end{equation}
Now \eqref{conv105} and \eqref{conv110} yield
$K^{ji}(\bfu_{\varepsilon_n})\nabla u^i_{\varepsilon_n}
\rightharpoonup K^{ji}(\bfu)\nabla u^i$ in $L^2(Q_T)^N$. The proof
of Lemma \ref{weak_convergence} is complete.
\end{proof}

\smallskip

By Lemma \ref{weak_convergence} we have
$$
\int_0^T \langle {\mathscr{T}}(\bfu_{\varepsilon_n}),\bfv
\rangle{\rm d}t \rightarrow \int_0^T \langle
{\mathscr{T}}(\bfu),\bfv \rangle{\rm d}t
$$
for every $\bfv \in L^2(0,T;\mathbb{V})$. Thus using the inequality
\eqref{ineq_10} it follows
\begin{eqnarray*}
0 &\leq& \lim_{n\rightarrow \infty}  \int_0^T \langle
{\mathscr{T}}(\bfu_{\varepsilon_n}),\bfv-\bfu_{\varepsilon_n}
\rangle{\rm d}t\\
&=& \lim_{n\rightarrow \infty} \left\{\int_0^T \langle
{\mathscr{T}}(\bfu_{\varepsilon_n}),\bfu-\bfu_{\varepsilon_n}
\rangle{\rm d}t + \int_0^T \langle
{\mathscr{T}}(\bfu_{\varepsilon_n}),\bfv-\bfu \rangle{\rm d}t
\right\}.
\end{eqnarray*}
Since $\bfu_{\varepsilon_n} \rightarrow \bfu$ a.e. in $Q_T$ and
${\mathscr{T}}(\bfu_{\varepsilon_n})\rightharpoonup
{\mathscr{T}}(\bfu)$ we arrive at
\begin{displaymath}
\int_0^T\langle {\mathscr{T}}(\bfu),\bfvarphi-\bfu \rangle \;{\rm
d}t \geq 0 \; \textmd{ for every } \bfvarphi\in
L^{2}(0,T;\mathcal{K}).
\end{displaymath}
This completes the proof of Theorem \ref{main_existence_theorem}.
\end{proof}

\subsection{The uniqueness of the solution}
\label{sec:main_result_uniqueness}

In this section we prove the uniqueness of the solution. In order to
do that, we assume the structure condition
\begin{equation}\label{con_uniq}
K^{ji}(\bfz)=0 \textmd{ for } j\neq i \textmd{ (i.e. $K^{ji}$ is a
diagonal matrix)}.
\end{equation}
It is convenient to denote $K^{j}=K^{jj}$. In addition to (A2) and
\eqref{con_uniq} we suppose
\begin{equation}
K^{j}(\bfz)=K^{j}(z^j) {\textmd{ and }} c_1\leq K^{j}(\xi)\leq c_2
\quad \forall \xi \in \mathbb{R}, \; j=1,\dots,m. \label{con_uniq2}
\end{equation}

\begin{thm}
Let {\rm (A1)--(A5)} be satisfied and
\eqref{con_uniq}--\eqref{con_uniq2} hold. Moreover, assume that
there exists the constant $C_L>0$ such that for all $\xi_1,\xi_2\in
\mathbb{R}$ and $\bfz_1,\bfz_2\in \mathbb{R}^m$ we have
($j=1,\dots,m$)
\begin{equation}\label{lipsch_cond}
\left\{
\begin{array}{rcll}
|K^j(\xi_1)-K^j(\xi_2)|&\leq& C_L |\xi_1-\xi_2|,
\\
|\bfe^j(\bfz_1)-\bfe^j(\bfz_2)|&\leq& C_L |\bfz_1-\bfz_2|,&
\\
|F^j(\bfx,t,\bfz_1)-F^j(\bfx,t,\bfz_2)|&\leq& C_L |\bfz_1-\bfz_2|&
\forall\, [\bfx,t] \in Q_T,
\\
|g^j(\bfx,t,\bfz_1)-g^j(\bfx,t,\bfz_2)|&\leq& C_L |\bfz_1-\bfz_2|&
\forall\, [\bfx,t] \in \Gamma_2 \times (0,T).
\end{array} \right.
\end{equation}
Then the variational solution to \eqref{eq1}--\eqref{eq5} is unique.
\end{thm}
\begin{proof}
Using Kirchhoff transformation ${\mathscr{K}}$ one transfers the
nonlinearities in the elliptic part to the parabolic term. Introduce
the new unknown variable $\bfh={\mathscr{K}}(\bfu)$,
${\mathscr{K}}:\mathbb{R}^m\rightarrow\mathbb{R}^m$,
\begin{equation}\label{kirchhoff_transform}
h^j(t,\bfx):=\int_{0}^{u^j(t,\bfx)}K^j(\xi){\rm{d}}\xi, \quad
j=1,\dots,m.
\end{equation}
Due to \eqref{con_uniq2} ${\mathscr{K}}$ is continuous and
increasing, and one-to-one with $\mathscr{K}^{-1}$
Lipschitz-continuous. Let $\bfh={\mathscr{K}}(\bfu)$ and
$\widetilde{\bfh}={\mathscr{K}}(\widetilde{\bfu})$, $\mathscr{K}$ is
defined by \eqref{kirchhoff_transform}, where $\bfu$ and
$\widetilde{\bfu}$ are two variational solutions  to
\eqref{eq1}--\eqref{eq5} on $Q_{T}$. Set
\begin{equation}\label{}
R^j:=(B^j\circ{\mathscr{K}}^{-1})(\bfh)
-(B^j\circ{\mathscr{K}}^{-1})(\widetilde{\bfh}), \quad j=1,\dots,m,
\end{equation}
and denote $\bfR=(R^1,\dots,R^m)$. Note that $\bfR\in
L^2(0,T;\mathbb{V}^*)$. By the Lax-Milgram theorem there is a
function $\bfw \in L^2(0,T;\mathbb{V})$ such that
\begin{equation}\label{eq300}
\int_0^T \langle \bfR,\bfphi \rangle {\rm d}t = \int_{Q_T} \nabla
w^j \cdot
 \nabla \phi^j \;{\rm d}Q_T
\end{equation}
for every $\bfphi \in L^2(0,T;\mathbb{V})$. Now we follow the idea
presented by Alt \& Luckhaus in \cite{AltLuckhaus1983}. We have
\begin{multline}\label{eq300b}
\frac{1}{\tau}\int_0^{\tau} \langle \bfR,\bfw \rangle {\rm d}s +
\frac{2}{\tau}\int_{\tau}^{t+\tau} \langle
\bfR(s)-\bfR(s-\tau),\bfw(s) \rangle {\rm d}s
\\
=\frac{1}{\tau}\int_{0}^{t} \langle
\bfR(s+\tau)-\bfR(s),\bfw(s+\tau) \rangle {\rm d}s -
\frac{1}{\tau}\int_{0}^{t} \langle \bfR(s),\bfw(s+\tau)-\bfw(s)
\rangle {\rm d}s
\\
+\frac{1}{\tau}\int_{t}^{t+\tau} \langle \bfR,\bfw\rangle {\rm d}s.
\end{multline}
In view of \eqref{eq300} we obtain
\begin{multline}\label{eq300c}
\frac{1}{\tau}\int_{0}^{t} \langle \bfR(s+\tau)-\bfR(s),\bfw(s+\tau)
\rangle {\rm d}s - \frac{1}{\tau}\int_{0}^{t} \langle
\bfR(s),\bfw(s+\tau)-\bfw(s) \rangle {\rm d}s
\\
=\frac{1}{\tau}\int_{Q_t} \left( \nabla w^j(s+\tau)-\nabla
w^j(s)\right) \cdot \nabla w^j(s+\tau) {\rm d}Q_t
\\
- \frac{1}{\tau}\int_{Q_t} \nabla w^j(s)\cdot \left( \nabla
w^j(s+\tau)- \nabla w^j(s)\right) {\rm d}Q_t
\end{multline}
and
\begin{equation}\label{eq300d}
\frac{1}{\tau}\int_{t}^{t+\tau} \langle \bfR,\bfw\rangle {\rm d}s =
\frac{1}{\tau}\int_{t}^{t+\tau}\int_{\Omega} \nabla w^j \cdot \nabla
w^j {\rm d}\bfx {\rm d}s .
\end{equation}
Now the equations \eqref{eq300b}--\eqref{eq300d}, taken together,
yield
\begin{multline}
\frac{1}{\tau}\int_0^{\tau} \langle \bfR,\bfw \rangle {\rm d}s +
\frac{2}{\tau}\int_{\tau}^{t+\tau} \langle
\bfR(s)-\bfR(s-\tau),\bfw(s) \rangle {\rm d}s
\\
=\frac{1}{\tau}\int_{Q_t} \left( \nabla w^j(s+\tau) - \nabla
w^j(s)\right) \cdot \left(\nabla w^j(s+\tau) - \nabla w^j(s)
\right){\rm d}Q_t
\\
+\frac{1}{\tau}\int_{t}^{t+\tau}\int_{\Omega} \nabla w^j \cdot
\nabla w^j {\rm d}\bfx {\rm d}s.
\end{multline}
Hence, as $\tau \rightarrow 0$, we get
\begin{equation}\label{eq301}
\int_0^t \langle \partial_s \bfR,\bfw \rangle {\rm d}s =
\frac{1}{2}\int_{\Omega} \nabla w^j(t) \cdot \nabla w^j(t) \;{\rm
d}\bfx.
\end{equation}
Moreover, we have
\begin{equation}\label{eq302}
\int_{Q_t}  \bfR\cdot (\bfh-\widetilde{\bfh}) \;{\rm d}Q_t =
\int_{Q_t} \nabla w^j \cdot \nabla (h^j-\widetilde{h}^j) \;{\rm
d}Q_t.
\end{equation}
Applying the Kirchhoff transformation to \eqref{eq41} and taking
$\bfvarphi=\bfh \pm \bfw$ one obtains
\begin{multline}\label{eq410}
\int_0^t \langle \partial_s (B^j\circ{\mathscr{K}}^{-1})(\bfh)
,w^j\rangle {\rm d}s +\int_{Q_t} \nabla h^j \cdot \nabla w^j \;{\rm
d}Q_t  +\int_{Q_t} \widehat{\bfe}^j(\bfh)\cdot \nabla w^j \;{\rm
d}Q_t
\\
=\int_{Q_{t}} \widehat{\bfF}(\bfx,s,\bfh) \cdot \bfw \;{\rm d}Q_t
+\int_0^t\!\!\!\int_{\Gamma_2}\widehat{\bfg}(\bfx,s,\bfh) \cdot \bfw
\;{\rm d}S_t.
\end{multline}
Here we denote
\begin{equation}\label{assum_20}
\begin{array}{rclll}
\widehat{\bfe}^j(\bfh)&=&{\bfe}^j({\mathscr{K}}^{-1}(\bfh)),
\\
\widehat{F}^j(\bfx,s,\bfh)&=&{F}^j(\bfx,s,{\mathscr{K}}^{-1}(\bfh)),
\\
\widehat{g}^j(\bfx,s,\bfh) &=&
{g}^j(\bfx,s,{\mathscr{K}}^{-1}(\bfh)).
\end{array}
\end{equation}
Writing \eqref{eq410} for $\bfh$ and $\widetilde{\bfh}$ and taking
the difference of both equations we get for $t\in(0,T)$
\begin{eqnarray}\label{eq303}
\int_0^t \langle \partial_s \bfR,\bfw\rangle {\rm d}s &+& \int_{Q_t}
\nabla(h^j-\widetilde{h}^j) \cdot \nabla w^j \;{\rm d}Q_t
\nonumber \\
&=&-\int_{Q_t} \left(
\widehat{\bfe}^j(\bfh)-\widehat{\bfe}^j(\widetilde{\bfh}) \right)
\cdot \nabla w^j \;{\rm d}Q_t
\nonumber\\
&&+\int_{Q_{t}} \left(
\widehat{\bfF}(\bfx,s,\bfh)-\widehat{\bfF}(\bfx,s,\widetilde{\bfh})
\right) \cdot \bfw \;{\rm d}Q_t
\nonumber \\
&&+\int_0^t\!\!\!\int_{\Gamma_2} \left(
\widehat{\bfg}(\bfx,s,\bfh)-\widehat{\bfg}(\bfx,s,\widetilde{\bfh})
\right)\cdot \bfw \;{\rm d}S_t .
\nonumber \\
\end{eqnarray}
Estimating each integral on the right-hand side and using
\eqref{lipsch_cond} together with the Young inequality one obtains,
consequently,
\begin{multline}\label{eq304}
\int_{Q_t} \left(
\widehat{\bfe}^j(\bfh)-\widehat{\bfe}^j(\widetilde{\bfh}) \right)
\cdot \nabla w^j \;{\rm d}Q_t
\\
\leq c_1 \delta \int_0^t
\|\bfh-\widetilde{\bfh}\|^2_{L^2(\Omega)^m}{\rm d}s + c_2 C(\delta)
\int_0^t \|\bfw\|^{2}_{W^{1,2}(\Omega)^m} {\rm d}s
\end{multline}
and in the similar way
\begin{multline}\label{eq305}
\int_{Q_{t}} \left(
\widehat{\bfF}(\bfx,s,\bfh)-\widehat{\bfF}(\bfx,s,\widetilde{\bfh})
\right) \cdot \bfw \;{\rm d}Q_t
\\
\leq c_1 \delta \int_0^t
\|\bfh-\widetilde{\bfh}\|^2_{L^2(\Omega)^m}{\rm d}s + c_2 C(\delta)
\int_0^t \|\bfw\|^{2}_{W^{1,2}(\Omega)^m} {\rm d}s.
\end{multline}
Further
\begin{multline}\label{eq306}
\int_0^t\!\!\!\int_{\Gamma_2}\!\!\! \left(
\widehat{\bfg}(\bfx,s,\bfh)-\widehat{\bfg}(\bfx,s,\widetilde{\bfh})
\right)\cdot \bfw \;{\rm d}S_t
\leq c_L \!\!\! \int_0^t \|\bfh-\widetilde{\bfh}\|_{L^2(\Gamma_2)^m}
\,\|\bfw\|_{L^2(\Gamma_2)^m} {\rm d}s
\\
\leq c_1 \delta \int_0^t
\|\bfh-\widetilde{\bfh}\|^2_{W^{1,2}(\Omega)^m}{\rm d}s + c_2
C(\delta) \int_0^t \|\bfw\|^2_{W^{1,2}(\Omega)^m} {\rm d}s .
\end{multline}
Hence, we can rewrite \eqref{eq303} using the above estimates
together with equations \eqref{eq301} and \eqref{eq302} to obtain
\begin{multline}\label{eq307}
\frac{1}{2}\int_{\Omega} |\nabla \bfw(t)|^2 \;{\rm d}\bfx +
\int_{Q_t}\bfR\cdot (\bfh-\widetilde{\bfh})
 \;{\rm d}S_t
 \\
\leq c_1 \delta \int_0^t
\|\bfh-\widetilde{\bfh}\|^2_{W^{1,2}(\Omega)^m}{\rm d}s + c_2
C(\delta) \int_0^t \|\bfw\|^2_{W^{1,2}(\Omega)^m} {\rm d}s .
\end{multline}
Note that if $\widetilde{\bfh}\neq\bfh$ then the second term on the
left in \eqref{eq307} is positive. Hence, provided we select
$\delta$ sufficiently small, we obtain the integral inequality
\begin{displaymath}
\|\bfw(t)\|^2_{\mathbb{V}} \leq c \int_0^t
\|\bfw(s)\|^2_{\mathbb{V}} \;{\rm d}s,
\end{displaymath}
for a.e. $0 \leq t \leq T$,
from which we obtain, using the technique of Gronwall's lemma,
$\bfw={\bf0}$.
Hence $\widetilde{\bfu}=\bfu$ and the uniqueness follows (recall
that the Kirchhoff transformation $\bfh={\mathscr{K}}(\bfu)$ is a
Lipschitz continuous one-to-one mapping).
\end{proof}

\begin{rem}
All results in our paper remain valid if one assumes the
nonhomogeneous Dirichlet boundary condition $\bfu=\bfu^D$ on
$\Gamma_1 \times (0,T)$, where $\bfu^D \in L^2(0,T;\mathbb{V})\cap
L^{\infty}(Q_T)$.


\end{rem}



\subsection*{Acknowledgment}
This outcome has been achieved with the financial support of the
Ministry of Education, Youth and Sports of the Czech Republic,
project No. 1M0579, within activities of the CIDEAS research centre.
Additional support from the grant 201/09/1544 provided by the Czech
Science Foundation is greatly acknowledged.

\end{document}